\documentclass[11pt,a4paper]{article}
\usepackage[utf8]{inputenc}
\usepackage[T1]{fontenc}
\usepackage{lmodern}
\usepackage{amsmath,amssymb,amsthm,mathtools}
\numberwithin{equation}{section}
\usepackage[margin=2.54cm]{geometry}
\usepackage{microtype}
\usepackage{cite}
\usepackage{xcolor}
\usepackage{color}
\usepackage{thmtools,thm-restate}
\definecolor{mycitegreen}{HTML}{00A99D}
\usepackage[
hypertexnames=false,
  colorlinks=true,
  linkcolor={blue!85!black},
  citecolor={red!85!black},
  urlcolor=mycitegreen
]{hyperref}
\usepackage{cleveref}
\newtheorem{theorem}{Theorem}[section]
\newtheorem{proposition}[theorem]{Proposition}
\newtheorem{lemma}[theorem]{Lemma}
\newtheorem{corollary}[theorem]{Corollary}
\newtheorem{conjecture}[theorem]{Conjecture}
\theoremstyle{remark}

\newcommand{\dmin}{\delta^0}

\newcommand{\Prob}{\mathbb{P}}
\newcommand{\symdiff}{\mathbin{\triangle}}

\title{Degree conditions for $k$-strong orientations of digraphs}
\author{
J{\o}rgen Bang-Jensen$^{1,2}$  \qquad Shuo Wei$^{1}$ \\[1mm]
\small $^1$School of Mathematics, Shandong University, Jinan 250100, China\\
\small $^2$Department of Mathematics and Computer Science,\\
\small University of Southern Denmark, Odense DK-5230, Denmark\\
\small \href{mailto:jbj@imada.sdu.dk}{\texttt{jbj@imada.sdu.dk}} \qquad  \href{mailto:shuowei@mail.sdu.edu.cn}{\texttt{shuowei@mail.sdu.edu.cn}} 
}
\date{}

\begin{document}
\maketitle
\vspace{-2em}
\begin{abstract}
Jackson and Thomassen conjectured that every $2k$-strong digraph contains a spanning $k$-strong oriented subdigraph. We prove sharp degree conditions for the existence of such a subdigraph. For every fixed positive integer $k$ and all sufficiently large $n$, every $n$-vertex digraph $D$ with $\dmin(D)\ge\lfloor(n+k-1)/2\rfloor$ admits a $k$-strong orientation. This threshold is best possible even for the weaker conclusion that $D$ itself is $k$-strong. We also prove a sharp Woodall-type analogue for every fixed positive integer $k$ and all sufficiently large $n$: if $d_D^+(x)+d_D^-(y)\ge n+2k-2$ for every missing arc $xy$, then $D$ admits a $k$-strong orientation, and this bound is again best possible. As a further consequence, we determine the sharp minimum total degree threshold. Finally, the semi-degree result also remains valid when $k\le\alpha n$ for every fixed $0<\alpha<0.094882\ldots$.
\end{abstract}



\section{Introduction}\label{sc1}

Notation not explicitly defined in this paper is consistent with \cite{Joergen2}. All digraphs considered here are finite and loopless, with at most one arc in each direction between two distinct vertices. A \textbf{digon} is a pair of opposite arcs, and an \textbf{oriented graph} is a digraph with no digons. A digraph is \textbf{$k$-strong} if it has at least $k+1$ vertices and remains strongly connected after the deletion of any set of at most $k-1$ vertices. Similarly, a digraph is \textbf{$k$-arc-strong} if it remains strongly connected after the deletion of any set of at most $k-1$ arcs.

For a digraph $D$, an \textbf{orientation of $D$} is a spanning subdigraph obtained by deleting exactly one arc from each digon and retaining every arc outside the digons. An orientation of $D$ that is $k$-strong is called a \textbf{$k$-strong orientation of $D$}. Equivalently, $D$ admits a $k$-strong orientation if and only if it contains a spanning $k$-strong oriented subdigraph: any such subdigraph can be extended to an orientation of $D$, and $k$-strongness is preserved under the addition of arcs.

For an undirected graph, an orientation is called \textbf{$k$-connected} if the resulting digraph is $k$-strong. The corresponding orientation problems for edge- and arc-connectivity are classical. Robbins~\cite{Robbins} proved that an undirected graph has a strong orientation if and only if it is $2$-edge-connected, and Nash-Williams~\cite{NashWilliams} showed that it has a $k$-arc-strong orientation if and only if it is $2k$-edge-connected. Boesch and Tindell~\cite{Boesch} extended Robbins's theorem to mixed multigraphs, while Frank~\cite{Frank} obtained a general orientation theorem with prescribed lower bounds on the indegrees of vertex sets. In the directed setting, Jackson~\cite{Jackson} proved that every $2k$-arc-strong digraph has a $k$-arc-strong orientation. This also follows from a general orientation theorem by Frank \cite{Frank1996}.

The analogous problem for vertex-connectivity is substantially more difficult. Thomassen~\cite{Thomassen1989} asked whether, for every positive integer $k$, there is an integer $f(k)$ such that every $f(k)$-connected undirected graph has a $k$-strong orientation. Writing $f(k)$ for the least such integer, Jord\'an~\cite{Jordan} proved $f(2)\le18$, and Thomassen~\cite{Thomassen2015} later established $f(2)=4$. Garamv\"olgyi, Jord\'an, Kir\'aly and Vill\'anyi~\cite{Garamvolgyi} proved the existence of $f(k)$ for every $k$, with $f(k)\le320k^2$. The conjectured value $f(k)=2k$ remains open in general.

Replacing every edge of an undirected graph by a digon gives a symmetric digraph, so the undirected problem is a special case of the following conjecture.

\begin{conjecture}[Jackson and Thomassen~\cite{Thomassen1989}]\label{conj1}
For every positive integer $k$, every $2k$-strong digraph contains a spanning $k$-strong oriented subdigraph.
\end{conjecture}

For $k=1$, Conjecture~\ref{conj1} follows from the arc-connectivity results above. For general digraphs, however, even the existence of an integer $r$ such that every $r$-strong digraph contains a spanning $2$-strong oriented subdigraph remains open~\cite{Zhou1}.

The purpose of this paper is to study dense degree conditions guaranteeing a $k$-strong orientation. For a vertex $v$, write $d_D(v)=d^+_D(v)+d^-_D(v)$ for its total degree. We use $\delta^{0}(D)=\min\{\delta^{+}(D),\delta^{-}(D)\}$ for the minimum semi-degree of $D$, where $\delta^+(D)$ and $\delta^-(D)$ are the minimum outdegree and minimum indegree of $D$, respectively. 
We also write $\delta(D)=\min_{v\in V(D)}\bigl(d_D^+(v)+d_D^-(v)\bigr)$ for the minimum total degree.

Our first result shows that, for fixed $k$ and large $n$, the sharp semi-degree threshold for $k$-strongness itself already guarantees a $k$-strong orientation.

\begin{restatable}{theorem}{semidegreethm}\label{thm1}
For every positive integer $k$, there exists $n_0=n_0(k)$ such that every digraph $D$ on $n\ge n_0$ vertices with $\dmin(D)\ge\lfloor(n+k-1)/2\rfloor$ admits a $k$-strong orientation.
\end{restatable}

The threshold in Theorem~\ref{thm1} cannot be lowered, even if one only asks that $D$ itself be $k$-strong.

\begin{proposition}\label{prop1}
Let $1\le k\le n-1$. Every $n$-vertex digraph $D$ with $\dmin(D)\ge\lfloor(n+k-1)/2\rfloor$ is $k$-strong, and this bound cannot be decreased by one.
\end{proposition}

Theorem~\ref{thm1} immediately gives a sharp total degree consequence.

\begin{corollary}\label{cor}
For every positive integer $k$, there exists $n_0=n_0(k)$ such that every digraph $D$ on $n\ge n_0$ vertices with $\delta(D)\ge\lfloor(3n+k-3)/2\rfloor$ admits a $k$-strong orientation. The bound is best possible even for the weaker conclusion that $D$ itself is $k$-strong.
\end{corollary}

Our second main result is a Woodall-type analogue. Woodall~\cite{Woodall} proved that every $n$-vertex digraph satisfying $d_D^+(x)+d_D^-(y)\ge n$ for every missing arc $xy$ contains a directed Hamilton cycle. Ore--Woodall-type degree conditions for linkedness were studied by Ferrara, Jacobson and Pfender~\cite{FJP}. Define $\sigma_2(D)=\min_{\substack{x\ne y\\xy\notin A(D)}}\bigl(d_D^+(x)+d_D^-(y)\bigr)$, with the convention $\sigma_2(D)=\infty$ if $D$ is complete symmetric.

\begin{restatable}{theorem}{woodallthm}\label{thm2}
For every positive integer $k$, there exists $n_0=n_0(k)$ such that every digraph $D$ on $n\ge n_0$ vertices with $\sigma_2(D)\ge n+2k-2$ admits a $k$-strong orientation.
\end{restatable}

The degree-sum threshold is also best possible.

\begin{proposition}\label{prop2}
Let $k\ge1$ and $n\ge2k+1$. There exists an $n$-vertex digraph $D$ with $\sigma_2(D)=n+2k-3$ that admits no $k$-strong orientation. Hence the bound in Theorem~\ref{thm2} cannot be decreased by one.
\end{proposition}

For $k=1$, Theorem~\ref{thm2} is exactly Woodall's Hamiltonicity threshold. For general $k$, its hypothesis already implies that $D$ is $2k$-strong, see Lemma~\ref{2k-strong}. Thus Theorem~\ref{thm2} confirms Conjecture~\ref{conj1} for digraphs satisfying a Woodall-type degree condition.

The two results also have immediate consequences for orientations of undirected graphs. Replacing every edge of a graph $G$ by a digon we obtain the complete biorientation $\stackrel{\leftrightarrow}{G}
$ of $G$. Theorem~\ref{thm1} implies that, for every fixed positive integer $k$ and all sufficiently large $n$, every $n$-vertex graph $G$ with $\delta(G)\ge\lfloor(n+k-1)/2\rfloor$ admits a $k$-connected orientation. Likewise, Theorem~\ref{thm2} implies that $d_G(x)+d_G(y)\ge n+2k-2$ for every pair of nonadjacent vertices $x,y$ guarantees a $k$-connected orientation. Both bounds are best possible.

Several partial results toward Conjecture~\ref{conj1} are known for classes related to tournaments. A digraph is \textbf{semicomplete} if every pair of distinct vertices is adjacent, and a \textbf{tournament} is an
oriented semicomplete digraph. Bang-Jensen and Jord\'an~\cite{Joergen3} proved that every $3$-strong semicomplete digraph on at least five vertices contains a spanning $2$-strong tournament. Wang, Qi and Yan~\cite{Wang} proved that every $5$-strong semicomplete digraph on at least nine vertices contains a spanning
$3$-strong tournament. More generally, Bang-Jensen and Jord\'an~\cite{Joergen3} conjectured that every $(2k-1)$-strong semicomplete digraph on at least $2k+1$ vertices contains a spanning $k$-strong tournament. This would be best possible as shown by an infinite family of $(2k-2)$-strong semicomplete digraphs with no such orientation \cite{Joergen3}.

Related linear connectivity bounds are known for several other structured classes of digraphs. In particular, for $k\ge2$, $(3k-2)$-strong connectivity suffices for locally semicomplete digraphs~\cite{Guo} and for
quasi-transitive digraphs~\cite{Joergen2}, see also~\cite{Bang1} for earlier results on locally semicomplete digraphs. More recently, Zhou, Bang-Jensen, Zhou and Yan~\cite{Zhou1} obtained bounds of $4k+1$ for extended semicomplete digraphs and $5k$ for semicomplete split digraphs. For related work on random strong orientations of graphs, see Aksoy and Horn~\cite{Aksoy}.

The proof of Theorem~\ref{thm1} also yields a quantitative extension in which $k$ grows linearly with $n$: the same conclusion holds whenever $k\le\alpha n$ for any fixed $0<\alpha<0.094882\ldots$. This is stated precisely in Proposition~\ref{semiprop}.

\medskip
\noindent\textbf{Proof strategy and organization.}
Section~\ref{sec2} collects the preliminaries and the sharpness constructions. In Section~\ref{sec3}, we prove Theorem~\ref{thm1} by orienting the digons independently and uniformly at random. The main point is that every nontrivial cut contains linearly many forward arcs, while only polynomially many cuts can have close to the minimum possible number of such arcs. The same estimate, applied after deleting fewer than $k$ vertices, yields both Theorem~\ref{thm1} and Proposition~\ref{semiprop}.

The proof of Theorem~\ref{thm2} in Section~\ref{sec4} uses a different random orientation. We decompose the digon graph into edge-disjoint paths and cycles and orient each path or cycle as a whole. The resulting balance at each vertex excludes highly unbalanced bad cuts, and a second low-cut counting argument handles the remaining cuts. Section~\ref{sec5} concludes with some remarks.

\section{Preliminaries}\label{sec2}

For a vertex $v$ of a digraph $D$, let $d_D^+(v)$ and $d_D^-(v)$ denote its outdegree and indegree. An arc is \textbf{single} if it is not in a digon. For disjoint sets $X,Y\subseteq V(D)$, let $e_D(X,Y)$ be the number of arcs from $X$ to $Y$, and let $e_D(X)$ be the number of arcs with both ends in $X$. The subdigraph induced by $X$ is denoted by $D[X]$, and $D-S=D[V(D)\setminus S]$. The \textbf{complete symmetric digraph} on a vertex set contains both arcs between every pair of distinct vertices. We write $X\symdiff Y$ for the symmetric difference of two sets.

An \textbf{ordered cut} of $D$ is a pair $(X,Y)$ of nonempty disjoint sets whose union is $V(D)$. The standard cut characterization states that $D$ is strong if and only if $e_D(X,Y)>0$ for every ordered cut $(X,Y)$. Consequently, provided $|V(D)|\ge k+1$, the digraph $D$ is $k$-strong if and only if this condition holds in $D-S$ for every $S\subseteq V(D)$ with $|S|\le k-1$. 

We next prove Proposition~\ref{prop1} and Proposition~\ref{prop2}.

\begin{proof}[Proof of Proposition~\ref{prop1}]
Suppose that $D$ satisfies $\dmin(D)\ge\lfloor(n+k-1)/2\rfloor$ but $D$ is not $k$-strong. Then there is a set $S\subseteq V(D)$ with $s=|S|\le k-1$ and an ordered cut $(A,B)$ of $D-S$ such that $e_D(A,B)=0$. Every $a\in A$ has outdegree at most $|A|+s-1$, and every $b\in B$ has indegree at most $|B|+s-1$. It follows that
\begin{align*}
\dmin(D)
&\le \min\{|A|,|B|\}+s-1\\
&\le \left\lfloor\frac{n-s}{2}\right\rfloor+s-1
=\left\lfloor\frac{n+s}{2}\right\rfloor-1
\le\left\lfloor\frac{n+k-1}{2}\right\rfloor-1,
\end{align*}
a contradiction.

To prove that the bound is sharp, partition an $n$-vertex set into $A,B,S$, where $|S|=k-1$, $|A|+|B|=n-k+1$ and $\bigl||A|-|B|\bigr|\le1$. Since $k\le n-1$, both $A$ and $B$ are nonempty. Let $D[A\cup S]$ and $D[B\cup S]$ be complete symmetric digraphs, with no arcs between $A$ and $B$. Then $\dmin(D)=\min\{|A|,|B|\}+k-2=\lfloor(n+k-1)/2\rfloor-1$, but $D-S$ is not strong, so $D$ is not $k$-strong.
\end{proof}

\begin{proof}[Proof of Proposition~\ref{prop2}]
Let $G$ be obtained from $K_{n-1}$ by adding one vertex $v$ adjacent to exactly $2k-1$ vertices of the clique, and let $D$ be the symmetric digraph obtained by replacing every edge of $G$ by a digon. Since $n\ge2k+1$, there is a vertex $y$ of the clique not adjacent to $v$. For every such $y$,
\[
d_D^+(v)+d_D^-(y)=(2k-1)+(n-2)=n+2k-3,
\]
and the same value is obtained for the missing arc $yv$. These are the only missing arcs of $D$, so $\sigma_2(D)=n+2k-3$. In every orientation of $D$, the vertex $v$ has total degree $2k-1$. Hence either its outdegree or its indegree is at most $k-1$. Since every $k$-strong digraph has minimum indegree and minimum outdegree at least $k$, no orientation of $D$ is $k$-strong. Thus the bound in Theorem~\ref{thm2} cannot be decreased by one.
\end{proof}

The following observation explains why Theorem~\ref{thm2} falls within the connectivity range of Conjecture~\ref{conj1}.

\begin{lemma}\label{2k-strong}
Let $D$ be a digraph on $n\ge2k+1$ vertices. If $\sigma_2(D)\ge n+2k-2$, then $D$ is $2k$-strong.
\end{lemma}

\begin{proof} Suppose otherwise. There is a set $S\subseteq V(D)$ with $s=|S|\le2k-1$ and an ordered cut $(A,B)$ of $D-S$ such that $e_D(A,B)=0$. For any $a\in A$ and $b\in B$, the arc $ab$ is missing, while $d_D^+(a)\le |A|+s-1$ and $d_D^-(b)\le |B|+s-1$. Hence $\sigma_2(D)\le d_D^+(a)+d_D^-(b)\le n+s-2\le n+2k-3$,
a contradiction.
\end{proof}

\section{Proofs of Theorem~\ref{thm1} and Corollary~\ref{cor}}\label{sec3}

\subsection{A random orientation lemma}\label{subsec1}

The \textbf{random orientation} of a digraph retains every single arc and, independently for each digon, retains one of its two arcs with equal probability $1/2$.

\begin{lemma}\label{randomlemma}
Let $m\ge64$, and let $H$ be an $m$-vertex digraph with $\dmin(H)\ge\lfloor m/2\rfloor$. Its random orientation $R$ satisfies
\[
\Prob(R\text{ is not strong})
\le 2^{-m}+530m^{128}2^{-\lfloor m/2\rfloor}.
\]
In particular, the failure probability tends to zero exponentially as $m\to\infty$.
\end{lemma}

\begin{proof}
For an ordered cut $(X,Y)$ of $H$, write $a=\min\{|X|,|Y|\}$. If there is a single arc from $X$ to $Y$, it is retained in $R$, so $e_R(X,Y)=0$ is impossible. Otherwise, the $e_H(X,Y)$ arcs from $X$ to $Y$ belong to distinct digons, all of which must be directed towards $X$ for the cut to have no forward arc in $R$. Hence $\Prob(e_R(X,Y)=0)\le2^{-e_H(X,Y)}$ in all cases.

We first bound $e_H(X,Y)$. If $|X|=a$, every vertex of $X$ has at least $\lfloor m/2\rfloor-a+1$ out-neighbours in $Y$. If $|Y|=a$, every vertex of $Y$ has at least $\lfloor m/2\rfloor-a+1$ in-neighbours in $X$. Thus, in either case,
\begin{equation}\label{eq_q}
e_H(X,Y)\ge a(\lfloor m/2\rfloor-a+1)\ge \lfloor m/2\rfloor,
\end{equation}
where the last inequality follows from $1\le a\le \lfloor m/2\rfloor$ and $a(\lfloor m/2\rfloor-a+1)-\lfloor m/2\rfloor=(a-1)(\lfloor m/2\rfloor-a)\ge0$.
By the cut characterization of strong connectivity, it remains to bound the probability that at least one ordered cut has no forward arc. We divide the cuts into two ranges.

\medskip
\noindent\textbf{Cuts with $e_H(X,Y)\ge2m$.}
There are fewer than $2^m$ ordered cuts in this range, each with failure probability at most $2^{-2m}$. The union bound therefore gives a total contribution of at most $2^m2^{-2m}=2^{-m}$.

\medskip
\noindent\textbf{Cuts with $e_H(X,Y)<2m$.}
Call these cuts \textbf{low cuts}. We shall show that there are at most $530m^{128}$ of them.

First, every low cut has a side of size at most seven, or both sides have sizes within eight of $m/2$. Indeed, if $a\le m/4$, then $\lfloor m/2\rfloor-a+1\ge m/4$, so~\eqref{eq_q} gives $2m>am/4$ and hence $a<8$. If $a>m/4$, then
\[
2m>e_H(X,Y)\ge a(\lfloor m/2\rfloor-a+1)
>\frac m4\left(\frac m2-a+\frac12\right),
\]
which implies $a>m/2-8$. 

There are at most $2\sum_{i=1}^7\binom mi\le14m^7$ ordered cuts with a side of size at most seven. It remains to count the low cuts whose two sides have sizes strictly between $m/2-8$ and $m/2+8$. Call them \textbf{nearly balanced}. If there are no such cuts, the desired count follows. Otherwise, fix one, denoted by $(A,B)$, and compare the remaining nearly balanced low cuts with it.

Let $M_A=|A|\cdot(|A|-1)-e_H(A)$ be the number of arcs missing from $H[A]$ relative to the complete symmetric digraph on $A$. Summing outdegrees over $A$ gives $e_H(A)+e_H(A,B)\ge|A|\cdot\lfloor m/2\rfloor$, and therefore
\[
M_A\le |A|\cdot(|A|-1-\lfloor m/2\rfloor)+e_H(A,B)
<8|A|+2m<6m+64\le10m.
\]
We used $|A|<m/2+8$, $|A|-1-\lfloor m/2\rfloor<8$ and $m\ge64$. The same argument, summing indegrees over $B$, shows that the number $M_B$ of missing arcs in $H[B]$ also satisfies $M_B<10m$.

Consider any nearly balanced low cut $(X,Y)$, and set $x=|A\cap X|$. In the complete symmetric digraph on $A$, there are $x(|A|-x)$ arcs from $A\cap X$ to $A\cap Y$. At most $M_A$ are missing in $H$, and all the others contribute to $e_H(X,Y)$. Thus $x(|A|-x)\le e_H(X,Y)+M_A<12m$. If $t=\min\{x,|A|-x\}\ge64$, then
\[
x(|A|-x)=t(|A|-t)
\ge64\frac{|A|}{2}
>64\left(\frac m4-4\right)
=16m-256\ge12m,
\]
a contradiction. The same argument applies inside $B$. Thus, on each of $A$ and $B$, the set $X$ contains fewer than $64$ vertices or omits fewer than $64$ vertices. It follows that $|X\symdiff Z|<128$ for some $Z\in\{\varnothing,A,B,V(H)\}$.

For each of these four choices of $Z$, there are at most $\sum_{i=0}^{128}\binom mi$ possibilities for $X$. Here the sum overcounts the possible symmetric differences, and $\binom mi=0$ when $i>m$. Including the cuts with a small side, the total number of low cuts is at most
\[
14m^7+4\sum_{i=0}^{128}\binom mi
\le14m^7+516m^{128}\le530m^{128}.
\]
Since each low cut has $e_H(X,Y)\ge \lfloor m/2\rfloor$ by~\eqref{eq_q}, their total contribution to the failure probability is at most $530m^{128}2^{-\lfloor m/2\rfloor}$. Adding the contribution $2^{-m}$ from the cuts with $e_H(X,Y)\ge2m$ proves the lemma.
\end{proof}

\subsection{Proof of Theorem~\ref{thm1} and a quantitative extension}\label{subsec2}

\semidegreethm*

\begin{proof}
%[Proof of Theorem~\ref{thm1}]
Let $T$ be the random orientation of $D$. We first obtain a bound valid whenever $n-k+1\ge64$.

For $S\subseteq V(D)$ with $s=|S|\le k-1$, set $m=n-s$. Then
\begin{align*}
\dmin(D-S)
&\ge\left\lfloor\frac{n+k-1}{2}\right\rfloor-s
=\left\lfloor\frac{n+k-1-2s}{2}\right\rfloor\\
&\ge\left\lfloor\frac{n-s}{2}\right\rfloor
=\left\lfloor\frac m2\right\rfloor,
\end{align*}
where the second inequality uses $s\le k-1$. The restriction $T-S$ has exactly the distribution of the random orientation of $D-S$. Lemma~\ref{randomlemma} therefore gives $\Prob(T-S\text{ is not strong})\le C n^{128}2^{-(n-s)/2}$ for an absolute constant $C$. Taking the union bound over all such sets $S$ yields
\begin{equation}\label{eq_semifailure}
\Prob(T\text{ is not $k$-strong})
\le C n^{128}\sum_{s=0}^{k-1}\binom ns\,2^{-(n-s)/2}.
\end{equation}
No independence between the events associated with different sets $S$ is needed.

For fixed $k$, the right-hand side of~\eqref{eq_semifailure} is at most $Ck n^{k+127}2^{-(n-k+1)/2}=o(1)$. For sufficiently large $n$, it is less than one, so there is a $k$-strong orientation of $D$.
\end{proof}

\begin{proof}[Proof of Corollary~\ref{cor}]
For every $v\in V(D)$, we have $d_D^+(v)\ge \delta(D)-d_D^-(v)
\ge \left\lfloor(3n+k-3)/2\right\rfloor-(n-1)
=\left\lfloor(n+k-1)/2\right\rfloor$.
Similarly, $d_D^-(v)\ge\lfloor(n+k-1)/2\rfloor$. Hence
$\dmin(D)\ge\lfloor(n+k-1)/2\rfloor$, and Theorem~\ref{thm1} applies.

For sharpness, partition the vertex set into $A,B,S$, where $|S|=k-1$, $|A|+|B|=n-k+1$ and $\bigl||A|-|B|\bigr|\le1$. Let $D[A\cup S]$ and $D[B\cup S]$ be complete symmetric digraphs, include every arc from $B$ to $A$, and include no arc from $A$ to $B$. Then $D-S$ is not strong. For $x\in A$ and $y\in B$, we have $d_D(x)=2(|A|-1)+2|S|+|B|=n+|A|+k-3$ and $d_D(y)=2(|B|-1)+2|S|+|A|=n+|B|+k-3$, while every vertex of $S$ has total degree $2(n-1)$. Hence
\[
\delta(D)=n+\min\{|A|,|B|\}+k-3
=n+k-3+\left\lfloor\frac{n-k+1}{2}\right\rfloor
=\left\lfloor\frac{3n+k-3}{2}\right\rfloor-1.
\]
Thus the bound cannot be decreased.
\end{proof}

The estimate~\eqref{eq_semifailure} also gives a range in which $k$ may grow with $n$.

\begin{proposition}\label{semiprop}
Let $h(x)=-x\log_2x-(1-x)\log_2(1-x)$, $0<x<1$, be the binary entropy function, and let $\alpha_0\in(0,1/2)$ be the unique solution of $h(\alpha_0)=(1-\alpha_0)/2$. Then $\alpha_0=0.094882\ldots$. For every fixed $0<\alpha<\alpha_0$, there exists $n_0=n_0(\alpha)$ such that, whenever $n\ge n_0$ and $1\le k\le\alpha n$, every $n$-vertex digraph $D$ with $\dmin(D)\ge\lfloor(n+k-1)/2\rfloor$ admits a $k$-strong orientation.
\end{proposition}

\begin{proof}
The function $h(x)-(1-x)/2$ is strictly increasing on $(0,1/2)$, with limits $-1/2$ at zero and $3/4$ at $1/2$, so $\alpha_0$ is well defined.

Fix $0<\alpha<\alpha_0$ and let $1\le k\le\alpha n$. For all sufficiently large $n$, we have $n-k+1\ge64$, so~\eqref{eq_semifailure} applies. Since $0<\alpha<1/2$, the binomial entropy bound gives
\[
\sum_{s=0}^{k-1}\binom ns
\le\sum_{s=0}^{\lfloor\alpha n\rfloor}\binom ns
\le2^{h(\alpha)n}.
\]
For completeness, the last inequality follows by expanding $1=(\alpha+(1-\alpha))^n$: for $s\le\alpha n$, the factor $\alpha^s(1-\alpha)^{n-s}$ is at least $\alpha^{\alpha n}(1-\alpha)^{(1-\alpha)n}=2^{-h(\alpha)n}$.

For every term in~\eqref{eq_semifailure}, we also have $n-s\ge(1-\alpha)n$. Hence
\[
\Prob(T\text{ is not $k$-strong})
\le C n^{128}2^{\left(h(\alpha)-(1-\alpha)/2\right)n}=o(1).
\]
The estimate is uniform over $1\le k\le\alpha n$, and the assertion follows.
\end{proof}

The same argument as in Corollary~\ref{cor} gives the corresponding total degree statement in the same range of $k$.

\section{Proof of Theorem~\ref{thm2}}\label{sec4}

We use a second cut-counting argument, but with a different random orientation. Rather than orienting the digons independently, we orient the digons coherently along paths and cycles, so that their contributions to the indegree and outdegree differ by at most one
at each vertex. This balance forces every cut with no forward arc to have two large sides. As before, only polynomially many cuts need to be considered, and each has no forward arc with exponentially small probability.

We use the following standard consequence of Euler's theorem.

\begin{lemma}\label{decomposelemma}
Every finite simple graph has an edge decomposition into simple paths and
cycles such that each vertex is an endpoint of at most one path.
\end{lemma}

\begin{proof}
It suffices to consider each nontrivial connected component separately. If all degrees are even, its edges can be decomposed into edge-disjoint cycles. Otherwise, pair the odd-degree vertices and add an auxiliary edge for each pair, allowing parallel edges. The resulting multigraph is Eulerian and hence has an Euler tour. Deleting the auxiliary edges from the tour gives edge-disjoint trails such that each odd-degree vertex is an endpoint
of exactly one trail and no other vertex is an endpoint.

For each such trail, repeatedly remove a simple cycle whenever a vertex is repeated. What remains is a simple path with the same endpoints, and the removed edges are decomposed into edge-disjoint cycles. This gives the required decomposition.
\end{proof}

Let $G_D$ be the \textbf{digon graph} of $D$, with vertex set $V(D)$ and an edge between two vertices precisely when they form a digon in $D$. Fix a decomposition of $G_D$ as in Lemma~\ref{decomposelemma}. Independently for each path or cycle, choose one of its two directions with probability $1/2$, and retain every single arc of $D$. Let $T$ denote the resulting random orientation of $D$.

For each vertex $v$, let $d_{\mathrm{dig}}^+(v)$ and $d_{\mathrm{dig}}^-(v)$ be the outdegree and indegree contributed by the oriented edges of $G_D$. Then
\begin{equation}\label{eqdiffer}
\bigl|d_{\mathrm{dig}}^+(v)-d_{\mathrm{dig}}^-(v)\bigr|\le1.
\end{equation}
Indeed, every cycle and every path on which $v$ is internal contributes equally to these degrees, and $v$ is an endpoint of at most one path.

We first prove a bipartite degree estimate for the cut bounds.

\begin{lemma}\label{orebipartite}
Let $F$ be a bipartite graph with nonempty parts $X,Y$ of sizes $a,b$, and let $q=e(F)$. Suppose that $0\le t\le a+b$ and $d_F(x)+d_F(y)\ge t$ whenever $x\in X$, $y\in Y$ and $xy\notin E(F)$. Then
$q\ge tab/(a+b)$.

\end{lemma}

\begin{proof}
The assertion is immediate if $q=ab$. 
Suppose that $q<ab$. Summing the assumed degree condition over all nonedges between $X$ and $Y$ gives $t(ab-q)\le \sum_{\substack{x\in X,\ y\in Y\\xy\notin E(F)}}\bigl(d_F(x)+d_F(y)\bigr).$
For each $x\in X$, there are $b-d_F(x)$ non-neighbours of $x$ in $Y$, and similarly each $y\in Y$ has $a-d_F(y)$ non-neighbours in $X$.
Hence
\begin{align*}
    t(ab-q)
    &\le \sum_{x\in X}\bigl(b-d_F(x)\bigr)d_F(x)+\sum_{y\in Y}\bigl(a-d_F(y)\bigr)d_F(y)\\
    &=(a+b)q-\sum_{x\in X}d_F(x)^2-\sum_{y\in Y}d_F(y)^2,
\end{align*}
where we used $\sum_{x\in X}d_F(x)=\sum_{y\in Y}d_F(y)=q$. By the Cauchy--Schwarz inequality,  $\sum_{x\in X}d_F(x)^2\ge\frac{q^2}{a}$ and  $\sum_{y\in Y}d_F(y)^2\ge\frac{q^2}{b}$. Therefore
$$t(ab-q)\le(a+b)q-\frac{q^2}{a}-\frac{q^2}{b}=(a+b)q\left(1-\frac{q}{ab}\right).$$
Since $ab-q=ab(1-q/ab)$, and $q<ab$, the term $1-q/(ab)>0$. Cancelling it from both
sides gives $tab\le(a+b)q$, and hence $q\ge tab/(a+b)$.

\end{proof}

Let $S\subseteq V(D)$ with $|S|\le k-1$, and let $(A,B)$ be an ordered cut of $D-S$. Since every single arc is retained in $T$,  $e_T(A,B)=0$ is possible only if there is no single arc from $A$ to $B$. We call a triple $(A,S,B)$ with this property \textbf{eligible}. For such a triple, every arc from $A$ to $B$ belongs to a digon, so $e_D(A,B)$ is also the number of edges between $A$ and $B$ in $G_D$.

\begin{lemma}\label{orecut}
Let $n\ge2k$ and $\sigma_2(D)\ge n+2k-2$. For every eligible triple $(A,S,B)$,
\begin{equation}\label{eq_boundq}
e_D(A,B)\ge (2k-|S|)\cdot\frac{|A|\cdot|B|}{|A|+|B|}.
\end{equation}
If an orientation $T$ of $D$ satisfies~\eqref{eqdiffer} and $e_T(A,B)=0$, then
$e_D(A,B)\le \min\{|A|,|B|\}\cdot(|S|+1)$ and
$\min\{|A|,|B|\}\ge(n-|S|)/(k+1)$.

\end{lemma}

\begin{proof}
Let $F$ be the bipartite subgraph of $G_D$ consisting of the edges between $A$ and $B$. If $x\in A$, $y\in B$ and $xy\notin E(F)$, eligibility implies that $xy\notin A(D)$. Since
\[
d_D^+(x)\le |A|-1+|S|+d_F(x)
\quad\text{and}\quad
d_D^-(y)\le |B|-1+|S|+d_F(y),
\]
%the Ore condition gives 
it follows from  $\sigma_2(D)\ge n+2k-2$ that
$d_F(x)+d_F(y)\ge2k-|S|$. Now $2k-|S|\le n-|S|=|A|+|B|$, so Lemma~\ref{orebipartite} yields~\eqref{eq_boundq}.

Suppose that $e_T(A,B)=0$, implying that every edge between $A$ and $B$ in $G_D$ is oriented from $B$ to $A$. Set $\Delta(v)=d_{\mathrm{dig}}^+(v)-d_{\mathrm{dig}}^-(v)$. In the sum of $\Delta(v)$ over $A$, the internal edges cancel, the edges between $A$ and $B$ contribute $-e_D(A,B)$, and the total contribution of the edges between $A$ and $S$ is at most $|A|\cdot|S|$ in absolute value. Together with~\eqref{eqdiffer}, this gives
\[
-|A|\le \sum_{v\in A}\Delta(v)\le -e_D(A,B)+|A|\cdot|S|,
\]
and hence $e_D(A,B)\le |A|\cdot(|S|+1)$.  Summing over $B$ similarly gives $e_D(A,B)\le |B|\cdot(|S|+1)$.

Assume without loss of generality that $|A|\le |B|$. Combining~\eqref{eq_boundq} with
$e_D(A,B)\le |A|\cdot(|S|+1)$ gives $(2k-|S|)|B|\le(|S|+1)(|A|+|B|)$, or equivalently $(2k-2|S|-1)|B|\le(|S|+1)|A|$. Since $|S|\le k-1$, we have $2k-2|S|-1\ge1$ and $|S|+1\le k$, and hence $|B|\le k|A|$. Therefore $n-|S|=|A|+|B|\le(k+1)|A|$, which gives $|A|\ge(n-|S|)/(k+1)$.
\end{proof}

Although the digons are no longer oriented independently, the paths and cycles in the decomposition are. This independence gives the following probability bound.

\begin{lemma}\label{prob}
For every eligible triple $(A,S,B)$ with $|S|\le k-1$,
\[
\Prob\bigl(e_T(A,B)=0\bigr)
\le 2^{-\lceil e_D(A,B)/(|S|+1)\rceil}.
\]
If $n\ge2k$, $\sigma_2(D)\ge n+2k-2$ and this event has positive probability, then
$e_D(A,B)/(|S|+1)\ge(n-|S|)/(k+1)$.
\end{lemma}

\begin{proof}
Consider a path or cycle $P$ in the fixed decomposition of $G_D$, oriented in a direction compatible with $e_T(A,B)=0$. Every crossing between $A$ and $B$ along $P$ must be from $B$ to $A$. Between consecutive crossings, $P$ must therefore return from $A$ to $B$ through $S$. These returns use distinct vertices of $S$, since $P$ is simple. A path thus has at most $|S|+1$ edges between $A$ and $B$. For a cycle, the same argument applies cyclically and gives at most $|S|$ such edges.

If the event $e_T(A,B)=0$ is possible, then at least
$\lceil e_D(A,B)/(|S|+1)\rceil$ paths or cycles of the decomposition therefore contain an edge between $A$ and $B$. Each has at most one direction compatible with the event. Their directions are chosen independently, proving the probability bound.

For the second assertion, Lemma~\ref{orecut} gives
$|A|,|B|\ge(n-|S|)/(k+1)$. Together with
$|A|+|B|=n-|S|$, this implies
$|A|\cdot|B|/(|A|+|B|)\ge k(n-|S|)/(k+1)^2$. Hence~\eqref{eq_boundq} yields
\[
\frac{e_D(A,B)}{|S|+1}
\ge\frac{2k-|S|}{|S|+1}\cdot\frac{k(n-|S|)}{(k+1)^2}
\ge\frac{n-|S|}{k+1},
\]
where the last inequality follows from $|S|\le k-1$.
\end{proof}

We next bound the number of cuts with $e_T(A,B)=0$. As in Section~\ref{subsec1}, we fix one cut with few forward arcs and show that the first side of every other such cut differs from one of four fixed sets in only a bounded number of vertices.

\begin{lemma}\label{orelowcuts}
For every fixed positive integer $k$, there is a constant $R_k$ such that the following holds for all sufficiently large $n$. Let $D$ be an $n$-vertex digraph with $\sigma_2(D)\ge n+2k-2$, and fix $S\subseteq V(D)$ with $|S|\le k-1$. There are at most $4\sum_{i=0}^{2R_k}\binom ni$ ordered cuts $(A,B)$ of $D-S$ such that $(A,S,B)$ is eligible, $e_D(A,B)<2kn$, and the event $e_T(A,B)=0$ has positive probability.
\end{lemma}

\begin{proof}
If there is no such cut, the assertion is immediate. Otherwise, fix one, say $(A,B)$. By Lemma~\ref{orecut},
\begin{equation}\label{linearsize}
|A|,|B|\ge\frac{n-|S|}{k+1}.
\end{equation}
For $X\in\{A,B\}$, let
$M_X=|X|\cdot(|X|-1)-e_D(X)$ be the number of arcs missing from $D[X]$.

There are $|A|\cdot|B|-e_D(A,B)$ missing arcs from $A$ to $B$. Summing the Woodall-type degree inequalities over these missing arcs gives 
\begin{align*}
(n+2k-2)\bigl(|A|\cdot|B|-e_D(A,B)\bigr)
&\le
\sum_{\substack{x\in A,\ y\in B\\ xy\notin A(D)}}
\bigl(d_D^+(x)+d_D^-(y)\bigr)\\
&\le
|B|\sum_{x\in A}d_D^+(x)
+|A|\sum_{y\in B}d_D^-(y)\\
&\le
|B|\Bigl(|A|\cdot(|A|-1)-M_A
        +|A|\cdot|S|+e_D(A,B)\Bigr)\\
&\quad+
|A|\Bigl(|B|\cdot(|B|-1)-M_B
        +|B|\cdot|S|+e_D(A,B)\Bigr).
\end{align*}
Using $|A|+|B|=n-|S|$, the right-hand side equals
\[
|A|\cdot|B|\,(n+|S|-2)-|B|M_A-|A|M_B
 +(n-|S|)\cdot e_D(A,B).
\]
Therefore, rearranging gives $|B|M_A+|A|M_B\le e_D(A,B)\cdot(2n+2k-|S|-2) -(2k-|S|)|A|\cdot|B|$.

For sufficiently large $n$, \eqref{linearsize} gives
$|A|,|B|\ge n/(2(k+1))$, and $2n+2k-|S|-2\le3n$. Since
$e_D(A,B)<2kn$, the preceding inequality, after discarding its nonpositive last term, gives $|B|M_A+|A|M_B<6kn^2$. Since $M_A,M_B\ge0$, we obtain $M_A<6kn^2/|B|\le12k(k+1)n$ and, similarly, $M_B<6kn^2/|A|\le12k(k+1)n$.

Let $(X,Y)$ be any other cut counted by the lemma, and set $x=|A\cap X|$. Of the $x(|A|-x)$ possible arcs from $A\cap X$ to $A\cap Y$, at most $M_A$ are missing. Every remaining arc is counted by $e_D(X,Y)$, which is less than $2kn$. Consequently, $x(|A|-x)<\bigl(12k(k+1)+2k\bigr)n$.
Put $C_k=12k(k+1)+2k$ and $R_k=\lceil4(k+1)C_k\rceil$. By~\eqref{linearsize}, $|A|\ge n/(2(k+1))$ for sufficiently large $n$. Thus, writing $r=\min\{x,|A|-x\}$, we have
\[
x(|A|-x)=r(|A|-r)\ge\frac{r|A|}{2}\ge\frac{rn}{4(k+1)},
\]
and hence $r<R_k$. The same argument inside $B$ shows that $X$ contains fewer than $R_k$ vertices of $B$ or omits fewer than $R_k$ vertices of $B$.

It follows that $X$ differs in fewer than $2R_k$ vertices from one of $\varnothing$, $A$, $B$ and $A\cup B$. There are at most $4\sum_{i=0}^{2R_k}\binom ni$ choices for $X$, and each determines $Y$. This proves the lemma.
\end{proof}

\woodallthm*

\begin{proof}
Fix $k$ and take $n$ sufficiently large. Let $T$ be the random orientation constructed above. If $T$ is not $k$-strong, then some set $S\subseteq V(D)$ with $|S|\le k-1$ and some ordered cut $(A,B)$ of $D-S$ satisfy $e_T(A,B)=0$. Such a triple is eligible.

Consider only triples for which this event has positive probability. By Lemma~\ref{orecut},
\[
e_D(A,B)\le (|S|+1)\cdot\min\{|A|,|B|\}
\le\frac{k(n-|S|)}2<2kn.
\]
Lemma~\ref{orelowcuts} therefore bounds the number of these cuts by $n^{O_k(1)}$ for each fixed $S$. There are at most $\sum_{i=0}^{k-1}\binom ni\le kn^{k-1}$ choices for $S$, so the total number of triples is still $n^{O_k(1)}$.

For each such triple, Lemma~\ref{prob} bounds the probability of $e_T(A,B)=0$ by $2^{-(n-k+1)/(k+1)}$. The union bound now gives
\[
\Prob(T\text{ is not $k$-strong})
\le n^{O_k(1)}2^{-(n-k+1)/(k+1)}=o(1).
\]
For sufficiently large $n$, this probability is less than one, so $D$ admits a $k$-strong orientation.
\end{proof}

\section{Remarks}\label{sec5}

\medskip
\noindent\textbf{Sparse spanning subdigraphs.}
For fixed $k$ and all sufficiently large $n$, Theorem~\ref{thm1} gives a $k$-strong orientation of $D$. Applying Mader's theorem~\cite{Mader} to this orientation yields a spanning $k$-strong oriented subdigraph with at most $2k(n-k)$ arcs. On the other hand, every $k$-strong digraph has minimum outdegree at least $k$, and hence at least $kn$ arcs. It would be interesting to determine whether the same semi-degree condition guarantees a spanning $k$-strong oriented subdigraph with $kn+o(n)$ arcs for fixed $k$. For related results on split digraphs, see~\cite{Zhou2}.

\medskip
\noindent\textbf{Meyniel-type degree conditions.}
It is natural to ask for an analogue of Theorem~\ref{thm2} in terms of the
total degrees of nonadjacent vertices. Unlike the Woodall-type degree condition,
however, such a condition does not by itself imply high vertex-connectivity:
it is vacuous for semicomplete digraphs. Thus an additional connectivity
assumption is necessary.

\begin{conjecture}\label{meynielconj}
Let $k\ge1$ and $n\ge2k+1$. Every $(2k-1)$-strong digraph $D$ on $n$ vertices satisfying $d_D(x)+d_D(y)\ge 2n+4k-5$ for every pair of nonadjacent vertices $x,y$ admits a $k$-strong orientation.
\end{conjecture}

For $k=1$, the degree condition in Conjecture~\ref{meynielconj} is precisely Meyniel's condition for hamiltonicity~\cite{Meyniel}. For symmetric digraphs, it reduces, by parity, to the sharp Woodall-type degree condition $d_G(x)+d_G(y)\ge n+2k-2$ for every pair of nonadjacent vertices in the underlying graph $G$.

The degree bound in Conjecture~\ref{meynielconj} is best possible. Indeed, consider the symmetric digraph $D$ from the construction in Proposition~\ref{prop2}. The exceptional vertex $v$ has $2k-1$ neighbours in the underlying graph, and hence $D$ is $(2k-1)$-strong. For every vertex $y$ nonadjacent to $v$, $d_D(v)+d_D(y)=2(2k-1)+2(n-2)=2n+4k-6$. However, as in the proof of Proposition~\ref{prop2}, every orientation of $D$ has indegree or outdegree at most $k-1$ at $v$, and therefore no orientation of $D$ is $k$-strong.

For semicomplete digraphs the degree condition is vacuous. Hence Conjecture~\ref{meynielconj} would in particular imply the conjecture of Bang-Jensen and Jord\'an that every $(2k-1)$-strong semicomplete digraph on at least $2k+1$ vertices contains a spanning $k$-strong tournament~\cite{Joergen3}. Moreover, their examples of $(2k-2)$-strong semicomplete digraphs with no spanning $k$-strong tournament show that the connectivity assumption in Conjecture~\ref{meynielconj} is also best possible.
\medskip

\medskip
\noindent\textbf{\large AI Declaration.}
ChatGPT was used in exploratory discussions of proof strategies. In particular, it suggested the quantitative extension in Proposition~\ref{semiprop} and the balanced path (cycle) orientation used in the proof of Theorem~\ref{thm2}. The authors independently verified all statements and proofs.


\begin{thebibliography}{99}
\small
\setlength{\itemsep}{0pt}

\bibitem{Aksoy}
S.~Aksoy and P.~Horn,
Graphs with many strong orientations,
\emph{SIAM J. Discrete Math.} \textbf{30} (2016), 1269--1282.

\bibitem{Bang1}
J.~Bang-Jensen,
Locally semicomplete digraphs: A generalization of tournaments,
\emph{J. Graph Theory} \textbf{14} (1990), 371--390.


\bibitem{Joergen}
J.~Bang-Jensen,
$k$-strong spanning local tournaments in locally semicomplete digraphs,
\emph{Discrete Appl. Math.} \textbf{157} (2009), 2536--2540.


\bibitem{Joergen2}
J.~Bang-Jensen and G.~Gutin,
\emph{Digraphs: Theory, Algorithms and Applications}, 2nd ed., Springer, London, 2009.


\bibitem{Joergen3}
J.~Bang-Jensen and T.~Jord\'an,
Spanning $2$-strong tournaments in $3$-strong semicomplete digraphs,
\emph{Discrete Math.} \textbf{310} (2010), 1424--1428.


\bibitem{Boesch}
F.~Boesch and R.~Tindell,
Robbins's theorem for mixed multigraphs,
\emph{Amer. Math. Monthly} \textbf{87} (1980), 716--719.


\bibitem{Frank}
A.~Frank,
On the orientation of graphs,
\emph{J. Combin. Theory Ser. B} \textbf{28} (1980), 251--261.

\bibitem{Frank1996}
A.~Frank,
\newblock Orientations of graphs and submodular flows,
\newblock \emph{Congr. Numer.} 113 (1996), 111--142.


\bibitem{Garamvolgyi}
D.~Garamv\"olgyi, T.~Jord\'an, C.~Kir\'aly and S.~Vill\'anyi,
Highly connected orientations from edge-disjoint rigid subgraphs,
\emph{Forum Math. Pi} \textbf{13} (2025), e11.


\bibitem{FJP}
M.~Ferrara, M.~Jacobson and F.~Pfender,
Degree conditions for $H$-linked digraphs,
\emph{Combin. Probab. Comput.} \textbf{22} (2013), 684--699.


\bibitem{Guo}
Y.~Guo,
Spanning local tournaments in locally semicomplete digraphs,
\emph{Discrete Appl. Math.} \textbf{79} (1997), 119--125.


\bibitem{Jackson}
B.~Jackson,
Some remarks on arc-connectivity, vertex splitting, and orientation in graphs and digraphs,
\emph{J. Graph Theory} \textbf{12} (1988), 429--436.


\bibitem{Jordan}
T.~Jord\'an,
On the existence of $k$ edge-disjoint $2$-connected spanning subgraphs,
\emph{J. Combin. Theory Ser. B} \textbf{95} (2005), 257--262.



\bibitem{Mader}
W.~Mader,
Minimal $n$-fach zusammenh\"angende Digraphen,
\emph{J. Combin. Theory Ser. B} \textbf{38} (1985), 102--117.

\bibitem{Meyniel}
M.~Meyniel,
Une condition suffisante d'existence d'un circuit hamiltonien dans un
graphe orient\'e,
\emph{J. Combin. Theory Ser. B} \textbf{14} (1973), 137--147.


\bibitem{NashWilliams}
C.~St.~J.~A.~Nash-Williams,
On orientations, connectivity and odd-vertex-pairings in finite graphs,
\emph{Canad. J. Math.} \textbf{12} (1960), 555--567.


\bibitem{Robbins}
H.~E.~Robbins,
A theorem on graphs, with an application to a problem of traffic control,
\emph{Amer. Math. Monthly} \textbf{46} (1939), 281--283.


\bibitem{Thomassen1989}
C.~Thomassen,
Configurations in graphs of large minimum degree, connectivity, or chromatic number,
\emph{Ann. New York Acad. Sci.} \textbf{555} (1989), 402--412.


\bibitem{Thomassen2015}
C.~Thomassen,
Strongly $2$-connected orientations of graphs,
\emph{J. Combin. Theory Ser. B} \textbf{110} (2015), 67--78.


\bibitem{Woodall}
D.~R.~Woodall,
Sufficient conditions for circuits in graphs,
\emph{Proc. London Math. Soc.} (3) \textbf{24} (1972), 739--755.


\bibitem{Wang}
K.~Wang, Y.~Qi and J.~Yan,
Spanning $3$-strong tournaments in $5$-strong semicomplete digraphs,
\emph{Discrete Math.} \textbf{347} (2024), 113664.


\bibitem{Zhou1}
J.~Zhou, J.~Bang-Jensen, T.~Zhou and J.~Yan,
Highly connected spanning oriented subdigraphs in generalizations of semicomplete digraphs,
\href{https://arxiv.org/abs/2607.17150v2}{arXiv:2607.17150v2} (2026).

\bibitem{Zhou2}
J.~Zhou, J.~Bang-Jensen and J.~Yan,
Sparse spanning $k$-strong oriented subdigraphs in split digraphs,
\href{https://arxiv.org/abs/2608.11578v1}{arXiv:2608.11578v1} (2026).

\end{thebibliography}
\end{document}